\documentclass{article}[a4paper,11pt]
\usepackage[english]{babel}
\usepackage{amsmath,amssymb,amsthm,enumerate,bbm,latexsym,a4wide}

\newcommand{\assign}{:=}
\newcommand{\cdummy}{\cdot}
\newcommand{\infixand}{\text{ and }}
\newcommand{\mathLaplace}{\Delta}
\newcommand{\mathcatalan}{C}
\newcommand{\mathd}{\mathrm{d}}
\newcommand{\nobracket}{}
\newcommand{\tmaffiliation}[1]{\\ #1}
\newcommand{\tmem}[1]{{\em #1\/}}
\newcommand{\tmop}[1]{\ensuremath{\operatorname{#1}}}

\newenvironment{enumerateroman}{\begin{enumerate}[i.] }{\end{enumerate}}

\newtheorem{theorem}{Theorem}[section]
\newtheorem*{theorem*}{Theorem}
\newtheorem{corollary}[theorem]{Corollary}
\newtheorem{lemma}[theorem]{Lemma}
\newtheorem{proposition}[theorem]{Proposition}
\newtheorem{remark}[theorem]{Remark}
\newtheorem{definition}[theorem]{Definition}

\begin{document}

\title{Coming up from $- \infty$ for KPZ via stochastic control }

\author{
  Nicolas Perkowski, Carlos Villanueva Mariz
  \tmaffiliation{FU Berlin}
}

\maketitle

\begin{abstract}
  We derive a lower bound, independent of the initial condition, for the
  solution of the KPZ equation on the torus through its representation as the
  value function of a (conditional) stochastic control problem. With the same
  techniques, we also prove a bound for its oscillation, again independent of
  initial conditions, from which a Harnack type inequality for the rough heat
  equation (on the torus) can be obtained.
\end{abstract}

\section{Introduction}

The Kardar-Parisi-Zhang (KPZ) equation
\[ \left( \partial_t - \frac{1}{2} \Delta \right) h (t, x) = \frac{1}{2} |
   \partial_x h (t, x) |^2 + \xi (t, x), \quad (t, x) \in [0, T] \times
   (\mathbb{R} / \mathbb{Z}), \]
where $\xi$ is a space-time white noise, is a prototypical singular stochastic
partial differential equation (SPDE) that is solved within the theory of
regularity structures~{\cite{Hairer2014}} or paracontrolled
distributions~{\cite{Gubinelli2015Paracontrolled}}. It is a universal model
for fluctuations in ``weakly asymmetric'' interface growth
models~{\cite{Hairer2018Class,Goncalves2014}} and a central model within the
KPZ universality class~{\cite{Quastel2015,Matetski2021}}.

The equation is singular because solutions only have the regularity of a
Brownian motion, and therefore the nonlinearity $| \partial_x h (t, x) |^2$ is
only well-defined after renormalization, considering formally $| \partial_x h
(t, x) |^2 - \infty$, and as a space-time distribution. The equation can be
formally related to the the stochastic heat equation via $w = e^h$ and
\[ \partial_t w = \frac{1}{2} \Delta w + w \xi, \]
if the product on the right hand side is interpreted in the It{\^o} sense
{\cite{Bertini1997}}. The first intrinsic formulation is due to
Hairer~{\cite{Hairer2013KPZ}} based on rough paths, and later there were also
intrinsic formulations based on regularity structures~{\cite{Friz2020}} and
paracontrolled distributions~{\cite{KPZreloaded}}.

By default, these pathwise theories only give solutions up to a possibly
finite explosion time, and also that only on compact space domains. But in
fact~{\cite{Hairer2013KPZ}} already contained an almost sure global existence
result, which was achieved by relating $h$ to the solution $w$ of the
stochastic heat equation, and using almost sure global existence for the
stochastic heat equation. However, unlike the rest of the theory, this step is
not pathwise and the null set where explosion occurs may depend on the initial
condition.

A pathwise non-explosion result was derived in~{\cite{KPZreloaded}} via a
stochastic control formulation which leads to a uniform bound, in combination
with a pathwise implementation of the Cole-Hopf transform, which turns the
uniform bound into a bound in H{\"o}lder spaces. This approach was extended
in~{\cite{Perkowski2019}} to treat the equation on $\mathbb{R}$ instead of on
the torus $\mathbb{T} = \mathbb{R} / \mathbb{Z}$, but it still relied on the
Cole-Hopf transform. The work~{\cite{zhangHJB}} treats a more general
equation, where the Cole-Hopf transform is no longer available, and derives
pathwise bounds for weighted uniform and H{\"o}lder norms by a priori
estimates in weighted spaces, and using a suitable transformation (Zvonkin
transform).

All the bounds mentioned above are ``linear'' in the sense that the estimate
for $h$ at a positive time depends on the initial condition $h (0)$. The
purpose of this work is to derive truly nonlinear bounds for $h$ that use the
nonlinearity as a coercive term which yields better bounds than those
available for the linear equation. There are several results of this type
available for singular SPDEs, starting from the influential work of
Mourrat-Weber~{\cite{comingdownweber}} on coming down from infinity for the
$\Phi^4_3$ equation, see for example
also~{\cite{Gubinelli2019Global,Chandra2023,Duch2023}}, where the solution
$\phi$ to a $\Phi^4$ type equation can be bounded at time $t > 0$ by a
constant that is independent of the initial condition. All these works rely on
the very strong coercivity provided by the term $- \phi^3$ in the $\Phi^4$ and
related equations.

To the best of our knowledge, there are no comparable results for the KPZ
equation in the literature. Note that KPZ preserves a constant shift, meaning
that if $h$ solves the KPZ equation with initial condition $h_0$, then $h + c$
solves the KPZ equation with initial condition $h_0 + c$, for any $c \in
\mathbb{R}$. Therefore, the formulation of a coming down from infinity result
is a bit more subtle than for the $\Phi^4$ equation. But we can take
inspiration from the deterministic equation: If $h$ is the viscosity solution
of the Hamilton-Jacobi equation
\[ \partial_t h (t, x) = \frac{1}{2} | \partial_x h (t, x) |^2, \quad h (0, x)
   = \bar{h} (x), \]
then by the Hopf-Lax formula {\cite{Evans2010}}, Section 10.3.4,
\[ h (t, x) = \max_{y \in \mathbb{R}} \left\{ \bar{h} (y) - \frac{t}{2} \left(
   \frac{x - y}{t} \right)^2 \right\}, \]
we have
\[ h (t, x) \geqslant - \frac{1}{2 t} x^2 \]
whenever $\bar{h} (0) \geqslant 0$. Moreover, we have the following one-sided
estimate for the second derivative of $h$,
\[ h (t, x + \varepsilon) + h (t, x - \varepsilon) - 2 h (t, x) \leqslant
   \frac{\varepsilon^2}{t}, \]
see~{\cite{Cannarsa2004}}, Theorem~1.6.1, which does not depend on $h_0$ and
which is also related to Oleinik's one-sided estimate for the derivative of
Burgers equation, see~{\cite{Dafermos2010}}, Theorem~11.2.1.

In our setting, even a one-sided bound on the second derivative of the KPZ
equation with space-time white noise is of course out of the question, because
$h$ only has Brownian regularity. Instead, here we will derive a lower bound
that is similar in spirit to the one from the Hopf-Lax formula, and a bound on
the H{\"o}lder norm of the derivative of $h (t)$ which only depends on $t$ and
$\xi$, but not on $h_0$. More precisely, our main results are:

\begin{theorem*}
  Let $\beta > 0$ and $\bar{h} \in \mathcal{C }^{\beta} (\mathbb{T})$ and let
  $h : [0, T] \times \mathbb{T}$ solve the KPZ equation
  \[ \left( \partial_t - \frac{1}{2} \Delta \right) h (t, x) = \frac{1}{2} |
     \partial_x h (t, x) |^2 + \xi (t, x), \quad h (0, x) = \bar{h} (x) . \]
  \begin{enumerateroman}
    \item For almost all realizations of $\xi$ and all $T > 0$ there exists a
    constant $C \equiv C (T, \xi)$ such that whenever $\bar{h} |_{[-
    \varepsilon, \varepsilon]} \geqslant 0 \nobracket$, we have
    \begin{equation}
      h (t, x) \geqslant \log \varepsilon - C \left( 1 + \frac{1}{t} \right),
      \quad  (t, x) \in (0, T] \times \mathbb{T} . 
    \end{equation}
    \item For almost all realizations of $\xi$ and all $T > 0$ and $\alpha \in
    \left( 0, \frac{1}{2} \right)$ there exists a constant $C \equiv C (t, T,
    \alpha, \xi)$ such that
    \begin{equation}
      | h (t, x) - h (t, 0) | \leqslant C | x |^{\alpha}, \quad (t, x) \in (0,
      T] \times \mathbb{T},
    \end{equation}
    independently of the initial condition, where $\sup_{t \in [\varepsilon,
    T]} C (t, T, \alpha, \xi) < \infty$for any $\varepsilon > 0$.
  \end{enumerateroman}
\end{theorem*}

The first claim is shown in Theorem~\ref{roughcomingup}, while the second one
is Theorem~\ref{roughoscitheorem}. For simplicity, we restrict our attention
to the periodic KPZ equation on the circle $\mathbb{T}$. But we expect that
similar results hold in weighted spaces for the KPZ equation on the real line.

One of the key ingredients for proving these results will be the variational
representation of $h$ proved in Section 7 of {\cite{KPZreloaded}}, which
formally reads
\[ h (t, x) = \underset{v}{\sup} \mathbb{E}_x \left[ \overline{h} (\gamma^v_t)
   + \int_0^t (\xi (t - s, \gamma_s^v) - \infty) \mathd s - \frac{1}{2}
   {\int^t_0}^{} v_{s }^2 \mathd s \right], \]
where $\gamma_s^v = x + \int^s_0 v_u \mathd u + B_s$, with $B$ a Brownian
motion, and we take the supremum over progressively measurable $\nu$.

\paragraph{Structure of the paper}After recalling some preliminary facts in
Section 2, Section 3 deals with the regular case (meaning $\xi \in C_b ([0, T]
\times \mathbb{R})$) in  R . In that case, the result can be obtained as a
consequence of Bou{\'e}-Dupuis' formula and an estimation of
\[ \log \mathbb{P}_x (B_t \in [- 1, 1]), \]
see Theorem \ref{smoothkpz}.

Section 4 is devoted to the rough setting. The rigorous representation of $h$
involves singular diffusions of the form
\begin{equation}
  \gamma_t = x + \int^t_0 b (s, \gamma_s) \mathd s + B_t, \label{singdiffu}
\end{equation}
where the drift $b$ is distributional with Besov-H{\"o}lder regularity $C_T 
\mathcal{C}^{- \frac{1}{2} -}$. Such a process will not be a functional of
Brownian motion, so a direct application of the Bou{\'e}-Dupuis formula is no
longer possible; however, we can exploit the structure of $\gamma$ as a
Dirichlet process (explored in detail in {\cite{kremp2023rough}}) to
``recover'' Girsanov's theorem and solve this difficulty. Everything then
boils down, as in the regular case, to estimating
\[ \log \mathbb{P}_x (\gamma_t \in [- \varepsilon, \varepsilon]) . \]
In order to do this, we will derive heat kernel estimates for $\gamma$. This
is based on an extension of Zvonkin's transform to the paracontrolled regime
($b \in C_T  \mathcal{C}^{- \alpha} (\mathbb{T})$ for $\alpha \in \left(
\frac{1}{2}, \frac{2}{3} \right)$) in a similar fashion to
{\cite{zhang2017heat}}. Consider $\Phi_{\lambda} (s, y) = y + u_{\lambda} (s,
y)$, where $u_{\lambda} : [0, T] \times \mathbb{R} \rightarrow \mathbb{R}$
solves the (paracontrolled) PDE
\begin{equation}
  \left( \partial_t + \frac{1}{2} \Delta + b \partial_x \right) u_{\lambda} =
  - b + \lambda u_{\lambda} . \label{parapdeINTRO}
\end{equation}
When $\lambda$ is big enough, $\Phi_{\lambda} (s, \cdummy)$ is a
diffeomorphism of  R  for each $s \in [0, T]$, and we will prove that $Y_t =
\Phi_{\lambda} (t, \gamma_t)$ solves (think formally of It{\^o}'s formula)
\[ Y_t = \Phi_{\lambda} (0, x) + \int^t_0 \lambda u_{\lambda} (s, \Phi^{-
   1}_{\lambda} (s, Y_s)) \mathd s + \int^t_0 \nabla \Phi_{\lambda} (s,
   \Phi_{\lambda}^{- 1} (s, Y_s)) \mathd B_s . \]
For $u_{\lambda}$ as in \eqref{parapdeINTRO}, such an SDE has Lipschitz
continuous in space (uniformly in time) drift, and uniformly elliptic,
H{\"o}lder continuous diffusion coefficient, so existence, uniqueness and heat
kernel estimates are available. As $\gamma_t = (\Phi_{\lambda} (t))^{- 1}
(Y_t)$, we thus get heat kernel estimates for \eqref{singdiffu}.

Lastly, in Section 5 we prove the second claim
(Theorem~\ref{roughoscitheorem}), and comment how it is related to Harnack's
inequality for the heat equation, both in the regular and rough setting.

\subsection{Preliminaries and notation}

Throughout the paper we will identify $\mathbb{T} \approx \left[ -
\frac{1}{2}, \frac{1}{2} \right]$.

\begin{definition}
  Let $\alpha \in \mathbb{R}$. The H{\"o}lder-Besov space
  $\mathcal{C}^{\alpha} (\mathbb{T})$ is defined by
  \[ \mathcal{C}^{\alpha} (\mathbb{T}) \equiv B_{\infty, \infty}^{\alpha}
     (\mathbb{T}) \assign \left\{ f \in \mathcal{S}' (\mathbb{T}) : \| f
     \|_{\alpha} : = \underset{i \geqslant - 1}{\sup} (2^{i \alpha} \|
     \Delta_i f \|_{L^{\infty} (\mathbb{T})}) < \infty \right\}, \]
  where $(\Delta_j)_{j \geqslant - 1}$ are the usual Paley-Littlewood
  projections associated with a dyadic partition of unity. We say that
  $\varphi \in \mathcal{C}^{\alpha -} (\mathbb{T})$ if $\varphi \in
  \mathcal{C}^{\alpha - \varepsilon} (\mathbb{T})$ for every $\varepsilon >
  0$.
\end{definition}

For $\alpha \in (0, 1)$ and a given Banach space we will denote by
$\mathcal{C}_T^{\alpha} X$ the space of $\alpha$-H{\"o}lder continuous
functions from $[0, T]$ to $X$, endowed with the norm
\[ {\| f \|_{\mathcal{C}_T^{\alpha} X}}  \assign \underset{t \in [0, T]}{\sup}
   \| f (t) \|_X + \underset{0 \leqslant s < t \leqslant T}{\sup} \frac{\| f
   (t) - f (s) \|_X}{| t - s |^{\alpha}} . \]
Similarly, the space $C_T X$ consists of the continuous functions from $[0,
T]$ to $X$, endowed with the norm
\[ \| f \|_{C_T X \assign} \underset{t \in [0, T]}{\sup} \| f (t) \|_X . \]
The following interpretation of the (regular noise) KPZ equation as the value
function of a stochastic control problem will be of central importance.

\begin{lemma}
  \label{smoothrep}Let $\bar{h} \in \mathcal{C}^{\beta} (\mathbb{R})$ for
  $\beta > 0$, $\xi \in C_b ([0, T] \times \mathbb{R})$ (i.e. continuous and
  bounded) and $h$ solve
  \[ \left( \partial_t - \frac{1}{2} \Delta \right) h = \frac{1}{2} |
     \partial_x h |^2 + \xi, \quad h (0) = \bar{h} . \]
  Then
  \[ h (t, x) = \underset{v}{\sup} \mathbb{E}_x \left[ \bar{h} (X^v_t) +
     \int_0^t \xi (t - s, X_s^v) \mathd s - \frac{1}{2} {\int^t_0}^{} v_{s }^2
     \mathd s \right], \]
  with $X_s^v = x + \int^s_0 v_u \mathd u + B_s .$
\end{lemma}

\begin{proof}
  This is a linear-quadratic control problem and the result is classical. We
  assume that $\bar{h} \in C^2 (\mathbb{R})$. The general case follows by an
  approximation argument. For $\bar{h} \in C^2$ the solution $h$ to the PDE is
  in $C^{1, 2} ([0, T] \times \mathbb{R})$ and therefore we can apply
  It{\^o}'s formula to $\bar{h} (X^v_t) = h (t - t, X^v_t)$ to obtain for a
  martingale $M$
  \begin{eqnarray*}
    \bar{h} (X^v_t) & = & h (t, x) + \int_0^t \left( \partial_s + \frac{1}{2}
    \Delta + v_s \partial_x \right) h (t - s, X^v_s) \mathd s + M_t\\
    & = & h (t, x) - \frac{1}{2} \int_0^t | \partial_x h (t - s, \cdummy) -
    v_s |^2 (X^v_s) \mathd s + \int_0^t \left( - \xi (t - s, X^v_s) +
    \frac{1}{2} v_s^2 \right) \mathd s + M_t,
  \end{eqnarray*}
  and thus
  \[ \mathbb{E}_x \left[ \bar{h} (X^v_t) + \int_0^t \xi (t - s, X_s^v) \mathd
     s - \frac{1}{2} {\int^t_0}^{} v_{s }^2 \mathd s \right] = h (t, x) -
     \frac{1}{2} \mathbb{E}_x \left[ \int_0^t | \partial_x h (t - s, \cdummy)
     - v_s |^2 (X^v_s) \mathd s \right], \]
  which is upper bounded by $h(t,x)$ and equality is obtained for $v_s = \partial_x h (t - s, X^v_s)$.
\end{proof}

Next we provide the statement of the Bou{\'e}-Dupuis formula. Remember that
for two probability measures $\mathbb{Q} \ll \mathbb{P}$, the relative entropy
is defined as
\[ H (\mathbb{Q} | \mathbb{P} \nobracket) = \mathbb{E} \left[ \frac{\mathd
   \mathbb{Q}}{\mathd \mathbb{P}} \log \frac{\mathd \mathbb{Q}}{\mathd
   \mathbb{P}} \right] = \mathbb{E}_{\mathbb{Q}} \left[ \log \frac{\mathd
   \mathbb{Q}}{\mathd \mathbb{P}} \right] . \]
\begin{theorem}
  $\left( \cite{dupuis}, \tmop{Theorem} 3.1 \right) .$ Let $B : [0, T]
  \rightarrow \mathbb{R}^d$ be a $\mathbb{P}$-Brownian motion and $F$ a
  bounded measurable functional on $C ([0, T] ; \mathbb{R}^d)$. Then
  \[ \begin{array}{lll}
       \log \mathbb{E}_{\mathbb{P}} [e^{F (B)}] & = & \underset{\mathbb{Q} \ll
       \mathbb{P}}{\sup} \{ \mathbb{E}_{\mathbb{Q}} [F (B)] - H (\mathbb{Q} |
       \mathbb{P} \nobracket) \}\\
       & = & \underset{v}{\sup}  \mathbb{E}_{\mathbb{P}} \left[ F \left( B +
       \int^{\cdot}_0 v_s \mathd s \right) - \frac{1}{2} \int^T_0 | v_s |^2
       \mathd s \right],
     \end{array} \]
  where in the second line $v$ is taken over all progressively measurable
  processes with respect to the augmented filtration generated by $B$.
\end{theorem}

We will mostly rely on the second identity, which allows us to pass from the
stochastic control problem to an optimization over probability measures.

For the rest of the paper, we fix $\alpha \in \left( \frac{1}{2}, \frac{2}{3}
\right)$.

\section{Lower bound}

\subsection{Warm-up: The case of bounded continuous forcing}

Let us start by assuming that $\xi \in C_b ([0, T] \times \mathbb{R})$, and we
consider the non-periodic KPZ equation on $[0, T] \times \mathbb{R}$ with
forcing $\xi$. We also assume that $\bar{h} |_{[- 1, 1]} \geqslant 0$, instead
of $h|_{[- \varepsilon, \varepsilon]} \geqslant 0$ as we will assume on the
torus.

\begin{theorem}
  \label{smoothkpz}Let $\bar{h} \in \mathcal{C }^{\beta} (\mathbb{R})$ be such
  that $\bar{h} |_{[- 1, 1]} \geqslant 0 \nobracket$. There exists a constant
  $C = C (T) > 0$ for which
  \[ h (t, x) \geqslant - \int_0^t \| \xi (s) \|_{C_b (\mathbb{R})}
     \mathd s - C \left( 1 + \frac{1_{| x | > 1} x^2}{t} \right),
     \quad  (t, x) \in [0, T] \times \mathbb{R} .  \]
\end{theorem}

\begin{proof}
  By the Bou{\'e}-Dupuis formula and Lemma \ref{smoothrep},
  \begin{equation}
    \begin{array}{lll}
      h (t, x) & = & \underset{v}{\sup} \mathbb{E}_x \left[ \bar{h} (X^v_t) +
      \int_0^t \xi (t - s, X_s^v) \mathd s - \frac{1}{2} {\int^t_0}^{} v_{s
      }^2 \mathd s \right]\\
      & = & \underset{\mathbb{Q} \ll \mathbb{P}_x}{\sup} \left\{
      \mathbb{E}_{\mathbb{Q}} \left[ \bar{h} (B_t) + \int_0^t \xi (t - s, B_s)
      \mathd s \right] - H (\mathbb{Q} | \mathbb{P}_x) \nobracket \right\} .
    \end{array} \label{smoothdupuis}
  \end{equation}
  Choosing $\mathbb{Q}_{t, x} = \mathbb{P}_x (\cdummy | B_t \in \nobracket [-
  1, 1])$, we obtain
  \[ \begin{array}{lll}
       H (\mathbb{Q}_{t, x} | \mathbb{P} \nobracket) & = & \mathbb{E}_x \left[
       \frac{\mathd \mathbb{Q}_{t, x}}{\mathd \mathbb{P}_x} \log \frac{\mathd
       \mathbb{Q}_{t, x}}{\mathd \mathbb{P}_x} \right] = \mathbb{E
       }_{\mathbb{Q}_{t, x}} \left[ \log \frac{1_{\{ B_t \in [- 1, 1]
       \}}}{\mathbb{P}_x (B_t \in [- 1, 1])} \right] \\
       & = & - \log
       \mathbb{P}_x (B_t \in [- 1, + 1])\\
       & = & - \log \mathbb{P}_0 (B_t \in [- x - 1, - x + 1])
     \end{array} \]
  If $x \in [- 1, 1]$, then $| [- x - 1, - x + 1] \cap [- \varepsilon,
  \varepsilon] | \geqslant \varepsilon$ for all $\varepsilon \in (0, 1)$, and
  thus
  \[ t \mapsto \mathbb{P}_0 (B_t \in [- x - 1, - x + 1]) \]
  is uniformly (in $x \in [- 1, 1]$) bounded away from $0$ on $(0, T]$. Hence,
  $H (\mathbb{Q}_{t, x} | \mathbb{P}_x \nobracket) \lessapprox 1$ for $x \in
  [- 1, 1]$. On the other hand, for $x \in [- 1, 1]^c$, we have by Jensen's
  inequality and the concavity of the logarithm with the heat kernel $p (t, x)
  = \frac{1}{\sqrt{2 \pi t}} e^{- x^2 / (2 t)}$
  \[ \begin{array}{lcl}
       H (\mathbb{Q}_{t, x} | \mathbb{P}  \nobracket_x) & = & - \log
       \mathbb{P}_0 (B_t \in [- x - 1, - x + 1])\\
       & = & - \log \int^{- x + 1}_{- x - 1} p (t, y) \mathd y\\
       & \leqslant & - \int^{- x + 1}_{- x - 1} \frac{1}{2} \log 2 p (t, y)
       \mathd y\\
       & = & \int^{- x + 1}_{- x - 1} \frac{1}{2} \left( \log 2 + \frac{1}{2}
       \log 2 \pi t + \frac{y^2}{2 t} \right) \mathd y\\
       & \overset{| x | > 1, t \in [0, T]}{\lesssim} & 1 + \frac{x^2}{t},
     \end{array} \]
  so overall
  \[ H (\mathbb{Q}_{t, x} | \mathbb{P} \nobracket_x) \lesssim 1 + \frac{1_{| x
     | > 1} x^2}{t} . \]
  Plugging this back in \eqref{smoothdupuis} and recalling that $\bar{h} |_{[-
  1, 1]} \geqslant 0$ yields
  \begin{eqnarray*}
    h (t, x) & \geqslant & \mathbb{E}_{\mathbb{Q}_{t, x}} \left[ \bar{h} (B_t)
    + \int_0^t \xi (t - s, B_s) \mathd s \right] - H (\mathbb{Q}_{t, x} |
    \mathbb{P} \nobracket_x)\\
    & \geqslant & \mathbb{E}_x \left[ \left. \bar{h} (B_t) + \int_0^t \xi (t
    - s, x + B_s) \mathd s \right| B_t \in [- 1, 1] \right] - C \left( 1 +
    \frac{1_{| x | > 1} x^2}{t} \right)\\
    & \geqslant & \mathbb{E}_x \left[ \left. \int_0^t \xi (t - s, B_s) \mathd
    s \right| B_t \in [- 1, 1] \right] - C \left( 1 + \frac{1_{| x | > 1}
    x^2}{t} \right)\\
    & \geqslant & - \int^t_0 \| \xi (s) \|_{C_b (\mathbb{R})} \mathd s - C
    \left( 1 + \frac{1_{| x | > 1} x^2}{t} \right),
  \end{eqnarray*}
  which concludes the proof.
\end{proof}

\subsection{The case of space-time white noise forcing}

\subsubsection{Singular diffusions and variational representation}

As hinted in the introduction, in order to extend Theorem \ref{smoothkpz} to
the KPZ equation with singular noise, we need a representation of its solution
which is analogous to \eqref{smoothdupuis} (see Theorem \ref{varirep}). Such a
representation will require making sense of diffusion equations of the form
\begin{equation}
  \gamma_t = x + \int^t_0 b (s, \gamma_s) \mathd s + B_t,
  \label{singulardiff2}
\end{equation}
where $b$ is a distribution. In the low regularity regime that we consider
here, martingale problem solutions to this problem were first constructed
in~{\cite{Delarue2016}} with rough path techniques, for scalar-valued
$\gamma$. The rough path approach was extended to multidimensional $\gamma$
with paracontrolled distributions in {\cite{chouk}}; in
{\cite{kremp2023rough}} a notion of weak solution was studied, and proved to
be equivalent to the martingale problem formulation. The solution $\gamma$ is
not a semimartingale, but a Dirichlet process, i.e. it can be decomposed into
a martingale and a process of vanishing quadratic variation.

Recall that $\alpha \in \left( \frac{1}{2}, \frac{2}{3} \right)$ is fixed
throughout the paper. For general $b \in C_T \mathcal{C}^{- \alpha}
(\mathbb{T})$, the equation \eqref{singulardiff2} is not well-posed without
including an enhancement in the following sense (Definition 3.6 in
{\cite{chouk}} and Section 6 of {\cite{kremp2023rough}}).

\begin{definition}
  $(\tmop{Enhanced} \tmop{drift})$\label{enhanceddrift} We say that b$\in C_T
  \mathcal{C}^{- \alpha} (\mathbb{T})$ is an $\tmop{enhanced} \tmop{drift}$ if
  there exists $B \in C_T \mathcal{C }^{1 - 2 \alpha} (\mathbb{T})$ such that
  $(b, B) \in \mathcal{X}^{\alpha}, \tmop{where}$
  \[ \mathcal{X}^{\alpha} : = \tmop{cl}_{C_T \mathcal{C}^{- \alpha} \times C_T
     \mathcal{C}^{1 - 2 \alpha}} \{ (\eta, \mathcal{J}^T (\partial_x \eta)
     \odot \eta)  | \nobracket \eta \in C_T \mathcal{C}^{\infty} (\mathbb{T})
     \}, \]
  where $\tmop{cl}$ means the closure in $C_T \mathcal{C}^{- \alpha}
  (\mathbb{T}) \times C_T \mathcal{C}^{1 - 2 \alpha} (\mathbb{T})$ of the set
  in brackets, $\mathcal{J}^T (u)$ is the solution of
  \[ \left( \partial_t + \frac{1}{2} \mathLaplace \right) \mathcal{J}^T (u) =
     u, \quad \mathcal{J}^T (u) (T) = 0, \]
  and $\odot$ denotes the paraproduct of two distributions (see Chapter 2 of
  {\cite{bahouri2011fourier}} for more details).
\end{definition}

\begin{definition}
  $(\tmop{Martingale} \tmop{solution})$.\label{martingalesol} Let $(\Omega,
  \mathcal{F}, (\mathcal{F}_t)_{t \in [0, T]}, \mathbb{P})$ be a filtered
  probability space and $v$ a progressively measurable process. For a
  distributional drift $b \in C_T \mathcal{C}^{- \alpha} (\mathbb{T})$,
  admitting an enhancement as in ({\cite{chouk}}, Definition 3.6), we say that
  an adapted and continuous stochastic process $\gamma$ is a
  $\tmop{martingale} \tmop{solution}$ to
  \begin{equation}
    \gamma_t = x + \int^t_0 (b (s, \gamma_s) + v_s) \mathd s + B_t
    \label{singulardiff}, \quad 0 \leqslant t \leqslant T,
  \end{equation}
  if $\mathbb{P} (\gamma_0 = x) = 1$ and whenever $\varphi_T \in
  \mathcal{C}^{\alpha} (\mathbb{T})$, $f \in C_T C (\mathbb{T})$ and $\varphi$
  solves the paracontrolled PDE
  \begin{equation}
    \left( \partial_t + \frac{1}{2} \mathLaplace + b \partial_x \right)
    \varphi = f, \quad \varphi (T) = \varphi_T \label{paraPDE}
  \end{equation}
  on $[0, T]$, then
  \[ \varphi (t, \gamma_t) - \int^t_0 (f (s, \gamma_s) + \partial_x \varphi
     (s, \gamma_s) v_s) \mathd s, \quad t \in [0, T] \]
  is a martingale.
\end{definition}

\begin{remark}
  \label{regphi}If $b \in C_T \mathcal{C}^{- \alpha} (\mathbb{T})$, the
  solution $\varphi$ to equation \eqref{paraPDE} belongs to the space $C_T 
  \mathcal{C}^{2 - \alpha} (\mathbb{T}) \cap C^{1 - \frac{\alpha}{2}}_T
  L^{\infty} (\mathbb{T})$, see Theorem 3.10 in {\cite{chouk}}.
\end{remark}

Next, we discuss the decomposition of $\gamma$ as a Dirichlet process. This
will help us derive an It{\^o} formula despite the lack of time regularity.

\begin{definition}
  $(\tmop{Dirichlet} \tmop{process}) .$Let $(\Omega, \mathcal{F},
  (\mathcal{F}_t)_{t \in [0, T]}, \mathbb{P})$ be a filtered probability
  space. A real valued process $D$ is called a {\tmem{Dirichlet process}}
  (resp. {\tmem{weak Dirichlet process}}) if it admits the decomposition
  \begin{equation}
    D = A + M, \label{dirichdecomp}
  \end{equation}
  where $M$ is a continuous local martingale and $A$ has zero quadratic
  variation (resp. $\langle A, N \rangle = 0$ for any continuous local
  martingale $N$). For any (weak) Dirichlet process, decomposition
  \eqref{dirichdecomp} is unique.
\end{definition}

Now we state the theorem relating weak and martingale solutions. In the weak
solution setting, $b$ is again required to admit a proper enhancement
({\cite{kremp2023rough}}, Definition 2.7).

\begin{theorem}
  \label{dirichletdecomp}$\left( \cite{kremp2023rough}, \tmop{Theorem} 5.10
  \right)$. Suppose $b \in C ([0, T], \mathcal{C}^{- \alpha} (\mathbb{T}))$
  and $v = 0$. Then equation \eqref{singulardiff} has a martingale solution
  which is unique in law.
  
  If we denote this law by $\mathbb{P}_x$, there exists a filtered probability
  $(\Omega, \mathcal{F}, (\mathcal{F}_t^X)_{t \in [0, T]}, \mathbb{P})$ and
  stochastic processes $X_t$, $A_t$ and $W_t$ satisfying
  \begin{enumerateroman}
    \item $(\mathcal{F}^X_t)_{t \in [0, T]}$ is the canonical filtration
    generated by $X$ (completed and right-continuous),
    
    \item $\tmop{Law} (X) = \mathbb{P}_x,$
    
    \item $(A)_{t \in [0, T]}$ is a zero quadratic variation process and
    $(W_t)_{t \in [0, T]}$ is a Brownian motion,
    
    \item $X_t = x + A_t + W_t$ almost surely for every $t \in [0, T]$.
  \end{enumerateroman}
\end{theorem}

Morally, $A_t = \underset{n \rightarrow \infty}{\lim} \int^t_0 b_n (s, X_s)
\mathd s$ in a suitable sense, for smooth approximations $(b_n)_{n \in
\mathbb{N}}$ of $b$. The precise formulation depends on the assumptions we
want to impose on $b$, and in general takes a more complicated form; see
Definition 3.1 of {\cite{kremp2023rough}}.

\begin{corollary}
  {\tmem{({\cite{dirichlet}}, Proposition~3.10)}}\label{itoformula} In the
  setting of Theorem \ref{dirichletdecomp}, let $\varphi \in
  \mathcatalan_b^{0, 1} ([0, T] ; \mathbb{R})$ (that is, bounded and
  continuous in time with one bounded derivative in space). Then $\varphi
  (\cdot, X_{\cdot})$ is a weak Dirichlet process with decomposition
  \[ \varphi (t, X_t) = \varphi (0, x) + \int^t_0 \partial_x \varphi (s, X_s)
     \mathd W_s + \tilde{A}_t, \]
  where $(\tilde{A}_t)_{t \in [0, T]}$ is such that $\langle \tilde{A}, N
  \rangle = 0$ for any continuous local martingale $N$.
  
  In particular, if $\varphi$ solves \eqref{paraPDE} with $f = 0$, by the
  uniqueness of the decomposition of a weak Dirichlet process, $\tilde{A} =
  0$.
\end{corollary}

Next, we use the decomposition of $(\varphi (t, X_t))_t$ as a weak Dirichlet
process to prove a martingale representation in the context of
Theorem~\ref{dirichletdecomp}. This does not follow from the classical
martingale representation theorem for Brownian motion, because we only know
that $X$ is a weak solution and it is not clear if $\mathcal{F}^X_t \subset
\mathcal{F}^W_t$.

\begin{proposition}
  \label{martingalerep}Again in the setting of Theorem \ref{dirichletdecomp},
  let $M$ be an $L^2$-martingale on the filtered probability space $(\Omega, \mathcal{F},
  (\mathcal{F}_t^X)_{t \in [0, T]}, \mathbb{P})$. Then, there exists a
  progressively measurable $u$ such that $\mathbb{E} \left[ \int^T_0 u^2_s
  \mathd s \right] < \infty$ and
  \[ M_t = \mathbb{E} [M_0] + \int^t_0 u_s \mathd W_s, \qquad t \in [0, T] .
  \]
\end{proposition}

\begin{proof}
  Let $\bar{\varphi} \in C_b (\mathbb{T})$ and $t \in [0, T]$. If we consider
  the solution on $[0, t] \times \mathbb{T}$ to
  \[ \left( \partial_t + \frac{1}{2} \mathLaplace + b \nabla \right) \varphi =
     0, \quad \varphi (t) = \bar{\varphi}, \]
  we have by Corollary \ref{itoformula} that
  \[ \bar{\varphi} (X_t) = \varphi (0, x) + \int^t_0 \partial_x \varphi (s,
     X_s) \mathd W_s . \]
  If instead we consider two times $0 \leqslant t_1 < t_2$ and
  $\bar{\varphi}_1, \bar{\varphi}_2 \in C_b (\mathbb{T})$,
  \begin{align*}
       &\bar{\varphi}_1 (X_{t_1}) \bar{\varphi}_2 (X_{t_2}) =
       \bar{\varphi}_1 (X_{t_1}) \left( \varphi^{t_2, \bar{\varphi}_2} (t_1,
       X_{t_1}) + \int^{t_2}_{t_1} \partial_x \varphi^{t_2, \bar{\varphi}_2}
       (s, X_s) \mathd W_s \right)\\
       & = \int^{t_2}_{t_1} \bar{\varphi}_1 (X_{t_1}) \partial_x
       \varphi^{t_2, \bar{\varphi}_2} (s, X_s) \mathd W_s + \varphi^{t_1,
       \bar{\varphi}_1 \varphi^{t_2, \bar{\varphi}_2} (t_1, \cdot)} (0, x)\\
       &\qquad + \int^{t_1}_0 \partial_x \varphi^{t_1, \bar{\varphi}_1
       \varphi^{t_2, \bar{\varphi}_2} (t_1, \cdot)} (s, X_s) \mathd W_s\\
       & = \int^{t_2}_0 \left( 1_{[0, t_1]} (s) \partial_x \varphi^{t_1,
       \bar{\varphi}_1 \varphi^{t_2, \bar{\varphi}_2} (t_1, \cdot)} (s, X_s) +
       1_{[t_1, t_2]} (s) \bar{\varphi}_1 (X_{t_1}) \partial_x \varphi^{t_2,
       \bar{\varphi}_2} (s, X_s) \right) \mathd W_s\\
       &\qquad + \varphi^{t_1, \bar{\varphi}_1 \varphi^{t_2, \bar{\varphi}_2}
       (t_1, \cdot)} (0, x),
  \end{align*}
  where we denote by $\varphi^{s, \psi}$ the solution on $[0, s] \times
  \mathbb{T}$ of
  \[ \left( \partial_t + \frac{1}{2} \mathLaplace + b \partial_x \right)
     \varphi = 0, \qquad \varphi (s) = \psi . \]
  This argument generalizes to $\bar{\varphi}_1 (X_{t_1}) \bar{\varphi}_2
  (X_{t_2}) \cdots \bar{\varphi}_n (X_{t_n})$, so the result follows by a
  monotone class argument as in the classical martingale representation
  theorem (see for instance Section 5.4 of {\cite{legall}}).
\end{proof}

Finally, we introduce the stochastic control representation for $h$. We denote
by $\mathfrak{p}\mathfrak{m}$ the family of progressively measurable
processes, over all possible filtered probability spaces, and for a given $v
\in \mathfrak{p}\mathfrak{m}$ (and thus implicitly a given filtered
probability space), the family $\mathfrak{M} (v, x)$ consists of all
martingale solutions to \eqref{singulardiff}.

\begin{theorem}
  $\left( \cite{KPZreloaded}, \tmop{Theorem} 7.13 \right) . \label{varirep}$
  Let $\overline{h} \in \mathcal{C}^{\beta} (\mathbb{T})$ for any
  $\beta > 0$ and
  \[ b (s, y) = \partial_x (Y + Y^{\vee} + Y^R) (t - s, y), \quad s \in [0,
     t], \]
  with appropriately chosen stochastic data (depending only on the white
  noise $\xi$, see Remark \ref{enhancednoise})
  \[ (Y, Y^{\vee}, Y^R) \in C_T  \mathcal{C}^{1 / 2 -} (\mathbb{T}) \times C_T
     \mathcal{C}^{1 -} (\mathbb{T}) \times C_T  \mathcal{C}^{3 / 2 -}
     (\mathbb{T}) . \]
  Then,
  \begin{equation}
    \begin{array}{l}
      (h - Y - Y^{\vee} - Y^R) (t, x)\\
      = \underset{v \in \mathfrak{p}\mathfrak{m}}{\sup} \underset{\gamma \in
      \mathfrak{M} (v, x)}{\sup} \mathbb{E} \left[ \overline{h} (\gamma_t) - Y
      (0, \gamma_t) + \int^t_0 | \nobracket \partial_x Y^R (t - s, \gamma_s)
      |^2 \mathd s - \frac{1}{2} \int^t_0 | v_s (\gamma) |^2 \mathd s \right],
    \end{array} \label{realformula}
  \end{equation}
  and the optimal $v$ is
  \[ v_s (\gamma) = \partial_x (h - Y - Y^{\vee} - Y^R) (t - s, \gamma_s) . \]
  For both this optimal $v$ and $v = 0$, the family $\mathfrak{M} (v, x)$ is
  non-empty and every $\gamma \in \mathfrak{M} (v, x)$ has the same law.
\end{theorem}

\begin{remark}
  \label{enhancednoise}In Theorem \ref{varirep}, $Y$ and $Y^{\vee}$ are the
  usual trees
  \[ \begin{array}{l}
       \left( \partial_t - \frac{\Delta}{2} \right) Y = \xi,\\
       \left( \partial_t - \frac{\Delta}{2} \right) Y^{\vee} = \frac{1}{2} |
       \partial_x Y |^2 - \infty,
     \end{array} \]
  where $Y^{\vee}$ needs to be renormalized, whereas $Y^R$ solves the
  paracontrolled PDE
  \[ \left( \partial_t - \frac{\Delta}{2} \right) Y^R = \frac{1}{2} |
     \partial_x Y^{\vee} |^2 + \partial_x Y \partial_x Y^{\vee} + \partial_x
     (Y + Y^{\vee}) \partial_x Y^R, \quad Y^R (0) = 0. \]
  The theorem is then proved by noting that $h^R \equiv h - Y - Y^{\vee} -
  Y^R$ solves the paracontrolled PDE
  \[ \left( \partial_t - \frac{\Delta}{2} \right) h^R = \frac{1}{2} |
     \partial_x Y^R |^2 + \partial_x (Y + Y^{\vee} + Y^R) \partial_x h^R +
     \frac{1}{2} | \partial_x h^R |^2, \quad h^R (0) = \overline{h} - Y (0),
  \]
  and using Definition \ref{martingalesol}.
\end{remark}

Formula \ref{realformula} involves drifts $v$, and we want to replicate the
argument we used for the regular equation, so an expression with measures (and
the ``$v = 0$ diffusion'') would be desirable. This is where the martingale
representation helps us; the proof of the following Lemma is reminiscent of
one of the inequalities in the Bou{\'e}-Dupuis formula.

\begin{lemma}
  \label{singdupuistheorem}Let $t \in [0, T]$, $x \in \mathbb{T}$ and let
  $\mathbb{P}_{t, x}$ denote the law of the unique martingale solution to
  \[ \gamma_s = x + \int^s_0 (\partial_x Y + \partial_x Y^{\vee} + \partial_x
     Y^R) (t - u, \gamma_u) \mathd u + B_s, \quad 0 \leqslant s \leqslant t.
  \]
  Then,
  \begin{equation}
    \begin{array}{l}
      \underset{v \in \mathfrak{p}\mathfrak{m}}{\sup} \underset{\gamma \in
      \mathfrak{M} (v, x)}{\sup} \mathbb{E} \left[ \overline{h} (\gamma_t) - Y
      (0, \gamma_t) + \int^t_0 | \nobracket \partial_x Y^R (t - s, \gamma_s)
      |^2 \mathd s - \frac{1}{2} \int^t_0 | v_s (\gamma) |^2 \mathd s
      \right]\\
      = \underset{\mathbb{Q} \ll \mathbb{P}_{t, x}}{\sup} \left\{
      \mathbb{E}_{\mathbb{Q}} \left[ \overline{h} (\gamma_t) - Y (0, \gamma_t)
      + \int^t_0 | \nobracket \partial_x Y^R (t - s, \gamma_s) |^2 \mathd s
      \right] - H (\mathbb{Q} | \mathbb{P}_{t, x} \nobracket) \right\} .
    \end{array} \label{singdupuis}
  \end{equation}
\end{lemma}

\begin{proof}
  To prove that the first line of \eqref{singdupuis} is greater than or equal
  to the second, we start by noting that since $\overline{h}$, $Y$ and
  $\partial_x Y^R \tmop{are} \tmop{all} \tmop{bounded} \tmop{functions},$the
  supremum in the second formula is uniquely attained by
  $\overline{\mathbb{Q}}$ defined as
  \[ \frac{\mathd \overline{\mathbb{Q}}}{\mathd \mathbb{P }_{t, x}}  =
     \frac{\exp \left( \overline{h} (\gamma_t) - Y (0, \gamma_t) + \int^t_0 |
     \nobracket \partial_x Y^R (t - s, \gamma_s) |^2 \mathd s
     \right)}{{\mathbb{E}_{\mathbb{P}_{t, x}}}  \left[ \exp \left(
     \overline{h} (\gamma_t) - Y (0, \gamma_t) + \int^t_0 | \nobracket
     \partial_x Y^R (t - s, \gamma_s) |^2 \mathd s \right) \right]}, \]
  see Proposition 2.5 in {\cite{dupuis}}.
  
  Now consider, for our $\mathbb{P}_{t, x},$the process $(X_s)_{s \in [0, t]}$
  and the basis $(\tilde{\Omega}, \widetilde{\mathcal{F}},
  (\widetilde{\mathcal{F}}^X_s)_{s \in [0, t]}, \widetilde{\mathbb{P}})$
  provided by Theorem \ref{dirichletdecomp}; we use the $\tilde{\cdummy}$
  notation to clarify that $\widetilde{\mathbb{P}}$ and
  $\widetilde{\mathbb{Q}}$ may be probability measures on a different
  measurable space than the canonical path space. Then we are done with the
  $(\geqslant)$ inequality, if for the measure
  \[ \frac{\mathd \widetilde{\mathbb{Q}}}{\mathd \widetilde{\mathbb{P}}}  =
     \frac{\exp \left( \overline{h} (X_t) - Y (0, X_t) + \int^t_0 | \nobracket
     \partial_x Y^R (t - s, X_s) |^2 \mathd s
     \right)}{{\mathbb{E}_{\widetilde{\mathbb{P}}}}  \left[ \exp \left(
     \overline{h} (X_t) - Y (0, X_t) + \int^t_0 | \nobracket \partial_x Y^R (t
     - s, X_s) |^2 \mathd s \right) \right]}, \]
  we can find a corresponding drift $v$ such that under
  $\widetilde{\mathbb{Q}}$, the process $(X_s)$ solves the $v$-martingale
  problem. To that end, define the bounded
  $(\widetilde{\mathcal{F}}_s^X)$-martingale
  \[ R_s = \mathbb{E} \left[ \left. \frac{\mathd
     \widetilde{\mathbb{Q}}}{\mathd \widetilde{\mathbb{P}}} \right|
     \widetilde{\mathcal{F}}_s^X \right] . \]
  By Proposition \ref{martingalerep}, there exists a progressively measurable
  $u$ for which $\mathbb{E}_{\widetilde{\mathbb{P}}} \left[ \int^t_0 u^2_s
  \mathd s \right] < \infty$ and
  \[ R_s = 1 + \int^s_0 u_u \mathd W_u . \]
  Observing that $R_s \geqslant c > 0$, again thanks to the boundedness of
  $h$, $Y$ and $\partial_x Y^R$, we can write for $\nu_s = u_s / R_s$
  \[ R_s = 1 + \int^s_0 \nu_u R_u \mathd W_u, \]
  and since
  \[ \mathbb{E}_{\widetilde{\mathbb{P}}} \left[ \int^t_0 \nu^2_s \mathd s
     \right] = \mathbb{E}_{\widetilde{\mathbb{P}}} \left[ \int^t_0
     \frac{u^2_s}{R^2_s} \mathd s \right] \leqslant c^{- 2}
     \mathbb{E}_{\widetilde{\mathbb{P}}} \left[ \int^t_0 u^2_s \mathd s
     \right] < \infty, \]
  we conclude that necessarily
  \[ R_s = \exp \left( \int^s_0 \nu_u \mathd W_u - \frac{1}{2} \int^s_0
     \nu^2_u \mathd u \right) . \]
  As $H (\widetilde{\mathbb{Q}} | \widetilde{\mathbb{P}}
  \nobracket)$=$\mathbb{E}_{\widetilde{\mathbb{Q}}} \left[ \frac{1}{2}
  \int^t_0 \nu^2_s \mathd s \right]$ (see the first part of the proof of
  Theorem 3.1 in {\cite{dupuis}}), all that remains to see is that under
  $\widetilde{\mathbb{Q}}$, the process $(X_s)_{s \in [0, t]}$ solves the
  $\nu$-martingale problem. For $\varphi$ and $f$ as in \eqref{paraPDE}, we
  know that
  \[ \varphi (s, X_s) - \int^s_0 f (u, X_u) \mathd u, \quad s \in [0, t], \]
  is a martingale under $\widetilde{\mathbb{P}}$. We also have $\varphi \in C
  _t  \mathcal{C}^{2 - \alpha} (\mathbb{T})$, so by Corollary
  \ref{itoformula},
  \[ \left\langle \varphi (\cdummy, X_.), \int^._0 \nu_u \mathd W_u
     \right\rangle_s = \int^s_0 \partial_x \varphi (u, X_u) \nu_u \mathd u, \]
  and by Girsanov's theorem
  \[ \varphi (s, X_s) - \int^s_0 (f (u, X_u) + \partial_x \varphi (u, X_u)
     \nu_u) \mathd u, \quad s \in [0, t], \]
  is a martingale under $\widetilde{\mathbb{Q}}$, as we wanted.
  
  The reverse inequality can be argued similarly by noting that for the
  optimal $v$ existence and uniqueness in law of solutions hold for the
  associated singular diffusion.
\end{proof}

\subsubsection{Zvonkin's transform}

We have seen thus far that, as in the case of smooth coefficients, we can
represent the solution to the KPZ equation by a stochastic control problem. In
order to derive similar bounds in our singular setting, we still require
estimates for the transition density of
\[ \gamma_s = x + \int^s_0 (\partial_x Y + \partial_x Y^{\vee} + \partial_x
   Y^R) (t - u, \gamma_u) \mathd u + B_s, \]
for fixed $t \in [0, T]$. Similarly as in {\cite{zhang2017heat}}, this is
where Zvonkin's transform comes into play. For the Zvonkin transform we need
to solve (paracontrolled) PDEs of the type
\begin{equation}
  \left( \partial_t + \frac{\Delta}{2} + b \partial_x \right) u_{\lambda} = f
  + \lambda u_{\lambda}, \quad u (t) = 0, \label{paraPDE2}
\end{equation}
where $b \in C_t  \mathcal{C}^{- \alpha} (\mathbb{T})$ and $f \in C_t
L^{\infty} (\mathbb{T}) \cup \{ - b \}$. The first step is to check that for
$\lambda$ big enough we can make $\| u_{\lambda} \|_{C^1 (\mathbb{T})}$ as
small as we want.

\begin{lemma}
  $\left( \cite{zhangHJB}, \tmop{Lemma} 3.4 \right) .$\label{schauder} Let
  $\alpha \in \left( \frac{1}{2}, \frac{2}{3} \right)$ and $\theta \in \left(
  1 + \frac{3 \alpha}{2}, 2 \right)$. If $b$, $f$ and $u_{\lambda}$ are as in
  \eqref{paraPDE}, there exists a $\lambda_0 > 0$ such that, for every
  $\lambda \geqslant \lambda_0,$
  \[ \| u_{\lambda} \|_{C_t  \mathcal{C}^{\theta - \alpha}} + \| u_{\lambda}
     \|_{C^{(\theta - \alpha) / 2}_t L^{\infty}} \leqslant (1 \vee
     \lambda)^{\frac{\theta}{2} - 1} c (\alpha, \theta, \| b \|_{C_t 
     \mathcal{C}^{- \alpha}}, \| f \|_{C_t  \mathcal{C}^{- \alpha}}, \| b
     \odot \nabla \mathcal{I}f \|_{C_t  \mathcal{C}^{1 - 2 \alpha}}), \]
  where $b \odot \nabla \mathcal{I}f$ should be formally understood as a limit
  of smooth approximations $b_n \odot \nabla \mathcal{I}f_n$ in $C_t 
  \mathcal{C}^{1 - 2 \alpha} (\mathbb{T})$.
\end{lemma}

We are ready to prove the main result of the section.

\begin{theorem}
  $(\tmop{Zvonkin}' s \tmop{transfom})$\label{zvonkintrans}. Let $f = - b$ in
  \eqref{paraPDE}. Then, there exists $\lambda_0 > 0$ such that for all
  $\lambda \geqslant \lambda_0$ and $s \in [0, t]$, the function
  $\Phi_{\lambda} (s, y) = y + u_{\lambda} (s, y)$ is a $C^1$-diffeomorphism
  of $\mathbb{R}$. Furthermore, for any such $\lambda$ and under
  $\mathbb{P}_{t, x}$, the process $Y_s \assign \Phi_{\lambda} (s, \gamma_s)$
  solves
  \begin{equation}
    Y_s = \Phi_{\lambda} (0, x) + \int^s_0 \tilde{b} (u, Y_u) \mathd u +
    \int^s_0 \tilde{\sigma} (u, Y_u) \mathd B_u, \label{zvonkinSDE}
  \end{equation}
  where $\tilde{b} (s, y) = \lambda u_{\lambda} (s, \Phi^{- 1}_{\lambda} (s,
  y))$ and $\tilde{\sigma} (s, y) = \nabla \Phi_{\lambda} (s,
  \Phi_{\lambda}^{- 1} (s, y))$.
\end{theorem}

\begin{proof}
  By Lemma \ref{schauder}, we can choose $\lambda > 0$ large enough so that
  $\sup_{s \in [0, t]} \| \nabla u_{\lambda} (s, \cdummy) \|_{\infty} < 1$/2,
  and hence
  \begin{equation}
    \frac{1}{2} | y - z | \leqslant | \Phi_{\lambda} (s, y) - \Phi_{\lambda}
    (s, z) | \leqslant \frac{3}{2} | y - z |, \quad \forall s \in [0, t],
    \label{lipscond}
  \end{equation}
  so the first claim follows.
  
  For the second one, let $(b_n)_{n \in \mathbb{N}} \subset C_t 
  \mathcal{C}^{\infty} (\mathbb{T})$ satisfy (remember Definition
  \ref{enhanceddrift})
  \[ \begin{array}{cr}
       b_n \rightarrow b & \tmop{in} C_t  \mathcal{C}^{\alpha} (\mathbb{T}),\\
       b_n \odot \nabla \mathcal{J} b_n \rightarrow b \odot \nabla \mathcal{J}
       b & \quad \tmop{in} C_t  \mathcal{C}^{1 - 2 \alpha} (\mathbb{T}),
     \end{array} \]
  and consider $u_{\lambda}^n$, $X^n$ the unique solutions to
  \[ \begin{array}{lr}
       \left( \partial_t + \frac{\Delta}{2} + b_n \nabla \right) u_{\lambda}^n
       = - b_n + \lambda u_{\lambda}^n, & \quad u_{\lambda}^n (t) = 0,\\
       \tmop{dX}_s^n = b_n (s, X_s^n) \mathd s + \mathd B_s, & X_0 = x.
     \end{array} \]
  If we let $\Phi_{\lambda}^n = \tmop{id}_{\mathbb{R}} + u_{\lambda}^n$,
  thanks to Lemma \ref{schauder} (and after possibly increasing $\lambda$),
  for big enough $n$ we can apply It{\^o}'s formula to obtain
  \[ \Phi_{\lambda}^n (s, X_s^n) \assign Y_s^n = \Phi^n_{\lambda} (0, x) +
     \int^s_0 \tilde{b}_n (u, Y^n_u) \mathd u + \int^s_0 \tilde{\sigma}_n (u,
     Y^n_u) \mathd B_u, \]
  with $\tilde{b}_n (s, y) = \lambda u^n_{\lambda} (s, (\Phi^n_{\lambda})^{-
  1} (s, y))$ and $\tilde{\sigma}_n (s, y) = \nabla \Phi^n_{\lambda} (s,
  (\Phi_{\lambda}^n)^{- 1} (s, y))$. Due to the uniform convergence of
  $\Phi^n_{\lambda}$, $\tilde{b}_n$ and $\tilde{\sigma}_n$ to
  $\Phi_{\lambda}$, $\tilde{b}$ and $\tilde{\sigma}$, which follows as well
  from Lemma \ref{schauder}, necessarily
  \[ Y^n \rightarrow Y \quad \tmop{weakly}, \]
  where $(Y_s)_{s \in [0, t]}$ is the unique weak solution of SDE
  \eqref{zvonkinSDE} ({\cite{stroock}}, Theorem 8.2.1; since $\nabla
  \Phi^n_{\lambda} = 1 + \nabla u^n_{\lambda}$, we have uniform ellipticity of
  the diffusion coefficient). On the other hand, denoting by
  $\mathcal{F}_{\lambda}^n$ the map
  \[ \begin{array}{llll}
       \mathcal{F}_{\lambda}^n : & C ([0, t] ; \mathbb{R}) & \rightarrow & C
       ([0, t] ; \mathbb{R})\\
       & (\omega (s))_{s \in [0, t]} & \mapsto & (\Phi^n_{\lambda} (s, \omega
       (s)))_{s \in [0, t]},
     \end{array} \]
  the weak convergence of $(X_s^n)_{s \in [0, t]}$ to the solution of the
  singular diffusion (see the proof of Theorem 4.3 in {\cite{chouk}}) and the
  convergence of $\Phi^n_{\lambda}$ guarantees
  \[ \mathcal{F}_{\lambda}^n (X^n) \rightarrow \mathcal{F}_{\lambda} (\gamma)
     \quad \tmop{weakly}, \]
  and therefore $\tmop{Law} (Y) = \tmop{Law} (\mathcal{F}_{\lambda}
  (\gamma))$, as we wanted.
\end{proof}

The last ingredient is the heat kernel estimate for SDE \eqref{zvonkinSDE}.

\begin{proposition}
  \label{propheat}$\left( \cite{gradientest}, \tmop{Theorem} 1.2 \right) .$
  Let $b, \sigma : \mathbb{R}^+ \times \mathbb{R}^n \rightarrow \mathbb{R}^n$
  be bounded, H{\"o}lder continuous in space functions such that $\sigma$ is
  uniformly elliptic. Then, the unique weak solution of
  \[ \mathd X_t = b (t, X_t) \mathd t + \sigma (t, X_t) d B_t, \quad t
     \geqslant 0, X_0 = 0, \]
  admits a transition density function enjoying the following two sided
  Gaussian estimates
  \begin{equation}
    C^{- 1} (\tau - s)^{- n / 2} e^{- c^{- 1} \frac{| x - y |^2}{\tau - s}}
    \leqslant p (s, x ; \tau, y) \leqslant C (\tau - s)^{- n / 2} e^{- c
    \frac{| x - y |^2}{\tau - s}} \label{heatestimates}
  \end{equation}
  for every $x, y \in \mathbb{R}$ and $0 \leqslant s < \tau \leqslant T$,
  where the constants $c, C > 0$ depend only on the regularity of the
  coefficients, the ellipticity constant of $\sigma$, the dimension $n$ and
  the value of $T$. Furthermore, the following gradient estimate also holds
  \begin{equation}
    | \nabla_x p (s, x ; \tau, y) | \leqslant C_1 (\tau - s)^{- 3 / 2} e^{-
    c_1 \frac{| x - y |^2}{\tau - s}} . \label{gradest}
  \end{equation}
\end{proposition}

The SDE \eqref{zvonkinSDE} satisfies the assumptions of the previous Theorem,
and therefore its transition density $p_Y (s, x ; \tau, y)$ satisfies
estimates of the type \eqref{heatestimates} and \eqref{gradest}. Indeed,
$\tilde{\sigma}$ and $\tilde{b}$ are both bounded and continuous; besides,
\[ \frac{1}{2} \leqslant 1 - \| \nabla u_{\lambda} \|_{C ([0, t] \times
   \mathbb{T})} \leqslant \tilde{\sigma} (s, y) \leqslant 1 + \| \nabla
   u_{\lambda} \|_{C ([0, t] \times \mathbb{T})} \leqslant \frac{3}{2}, \]
and $\tilde{\sigma} (s, y) = \nabla \Phi_{\lambda} (s, \Phi_{\lambda}^{- 1}
(s, y))$ has the required regularity: $\Phi^{- 1}_{\lambda} (s, y) \equiv
(\Phi_{\lambda} (s))^{- 1} (y)$ is Lipschitz continuous in space (due to
\ref{lipscond}) so the H{\"o}lder regularity in space of $\tilde{\sigma}$
follows from that of $\nabla \Phi_{\lambda} $ (and similarly for $\tilde{b}$,
which is even more regular).

\begin{corollary}
  \label{singkernelcorol}The transition density $p_b$ of the singular
  diffusion
  \[ \mathd \gamma_s = b (s, \gamma_s) \mathd s + \mathd B_s \]
  satisfies for suitable $c', C' > 0$ and for $t \in [0, T]$,
  \begin{equation}
    {C'}^{- 1} t^{- 1 / 2} e^{{- c'}^{- 1} \frac{| x - y |^2}{2 t}} \leqslant
    p_b (0, x ; t, y) \leqslant C' t^{- 1 / 2} e^{- c' \frac{| x - y |^2}{2
    t}} \label{singkernel}
  \end{equation}
  and also
  \begin{equation}
    | \nabla_x p (0, x ; t, y) | \leqslant C'_1 t^{- 1} e^{- c'_1 \frac{| x -
    y |^2}{t}} . \label{singgradest}
  \end{equation}
\end{corollary}

\begin{proof}
  By Theorem \ref{zvonkintrans},
  \[ p_b (0, x ; t, y) = p_Y (0, \Phi_{\lambda} (0, x) ; t, \Phi_{\lambda}
     (t, y)) | \nabla \Phi_{\lambda} (t, y) | . \label{kernelexp} \]
  On one hand, we have that
  \[ \frac{1}{2} \leqslant | \nabla \Phi_{\lambda} (t, y) | \leqslant
     \frac{3}{2}, \]
  and on the other, thanks to the Lipschitz continuity in space and
  $\frac{1}{2}$-H{\"o}lder continuity in time of $\Phi_{\lambda}$,
  \[ \begin{array}{lll}
       | \Phi_{\lambda} (0, x) - \Phi_{\lambda} (t, y) |^2 & \gtrsim & |
       \Phi_{\lambda} (t, x) - \Phi_{\lambda} (t, y) |^2 - | \Phi_{\lambda}
       (0, x) - \Phi_{\lambda} (t, x) |^2\\
       & \gtrsim & | x - y |^2 - t,
     \end{array} \]
  and thus the result follows from Proposition \ref{propheat}. Noting that
  \[ \nabla_x p_b (0, x ; t, y) = \nabla_x p_Y (0, \Phi_{\lambda} (0, x) ; t,
     \Phi_{\lambda} (t, y)) \nabla \Phi_{\lambda} (0, x) | \nabla
     \Phi_{\lambda} (t, y) |, \]
  \eqref{singgradest} is obtained similarly.
  
  \ 
\end{proof}

\subsubsection{Proof of the lower bound}

We have everything we need to prove the main theorem, that is the analogue of
the lower bound in Theorem~\ref{smoothkpz} in the singular case. Since we are
now on the torus, we need to work with the ``periodized'' singular diffusion,
which has transition density
\[ p_{\mathbb{T}, b} (0, x ; t, y) = 2 \pi \sum_{k \in \mathbb{Z}} p_b (0, x ;
   t, y + 2 \pi k), \quad (t, x, y) \in \mathbb{R}^+ \times \mathbb{T} \times
   \mathbb{T}, \]
but we will only use the fact that for $x, y \in \mathbb{T}$ (and
independently of $t$),
\begin{equation}
  p_{\mathbb{T}, b} (0, x ; t, y) \gtrsim p_b (0, x ; t, y) .
  \label{geqkernels}
\end{equation}
\begin{theorem}
  \label{roughcomingup}Let $\bar{h} \in \mathcal{C }^{\beta} (\mathbb{T})$ be
  such that $\bar{h} |_{[- \varepsilon, \varepsilon]} \geq 0 \nobracket$.
  There exists a constant
  \[ C \equiv C (T, Y, Y^{\vee}, Y^R) \]
  for which
  \[ h (t, x) \geq \log \varepsilon - C \left( 1 + \frac{1}{t} \right), \quad 
     (t, x) \in (0, T] \times \mathbb{T} .  \]
\end{theorem}

\begin{proof}
  By \eqref{realformula} and \eqref{singdupuis}, we have as in Theorem
  \ref{smoothkpz} that
  \[ \begin{array}{lll}
       h (t, x) & \geqslant & - \| Y + Y^{\vee} + Y^R \|_{\infty} - \| Y
       \|_{\infty} - \int^T_0 \| \partial_x Y^{^R} (s) \|_{C_b (\mathbb{T})}^2
       \mathd s + \log \mathbb{P}_{t, x} (\gamma_t \in [- \varepsilon,
       \varepsilon]) .
     \end{array} \]
  Now, arguing as in Theorem \ref{smoothkpz}
  \[ \begin{array}{rcl}
       \log \mathbb{P}_{t, x} (\gamma_t \in [- \varepsilon, \varepsilon]) &
       \geqslant & \log \int^{\varepsilon}_{- \varepsilon} p_{\mathbb{T}, b}
       (0, x ; t, y) \mathd y \\& \overset{\eqref{geqkernels}}{\gtrsim} & \log
       \int^{\varepsilon}_{- \varepsilon} p_b (0, x ; t, y) \mathd y\\
       &  \overset{\eqref{singkernel}}{\gtrsim}& \log \frac{1}{2
       \varepsilon} \int^{\varepsilon}_{- \varepsilon} {2 \varepsilon C'}^{-
       1} t^{- 1 / 2} e^{{- c'}^{- 1} \frac{| x - y |^2}{t}} \mathd y\\
       & \overset{x \in \mathbb{T}, t \in [0, T]}{\gtrsim} & \log
       \varepsilon - C_T ( 1 + \frac{1}{t} ),
     \end{array} \]
  as we wanted.
\end{proof}

\section{Modulus of continuity}

With the same idea, we may prove one more interesting result: that the modulus
of continuity of the KPZ equation comes down from infinity, without any
quantitative requirements on the initial condition. This is the content of
Theorems \ref{oscitheorem} and \ref{roughoscitheorem}.

\subsection{Warm-up: The case without forcing}

In the following Theorem we denote by $\mathcal{N}_{\mathbb{T}} (t, x)$ the
law of the Brownian motion on the torus starting from $x \in \mathbb{T}$, at
time $t$ (i.e. the nornal distribution on the torus with variance $t$).

\begin{theorem}
  \label{oscitheorem}Let $\beta > 0$ and $\bar{h} \in \mathcal{C}^{\beta}
  (\mathbb{T})$ and $\xi = 0$. Then, with $C_T$ independent of $\bar{h}$,
  \begin{equation}
    | h (t, x) - h (t, 0) | \leqslant C_T \left( 1 + \frac{1}{t} \right) | x
    |, \quad (t, x) \in (0, T] \times \mathbb{T} . \label{oscillation}
  \end{equation}
\end{theorem}

\begin{proof}
  We have
  \[ \begin{array}{lll}
       h (t, x) - h (t, 0) & = & \underset{\mathbb{Q} \ll \mathbb{P}_{t,
       x}}{\sup} \left\{ \mathbb{E}_{\mathbb{Q}} \left[ \overline{h} (B_t)
       \right] - H (\mathbb{Q} | \mathbb{P}_x \nobracket) \right\} -
       \underset{\mathbb{Q} \ll \mathbb{P}_{t, 0}}{\sup} \left\{
       \mathbb{E}_{\mathbb{Q}} \left[ \overline{h} (B_t) \right] - H
       (\mathbb{Q} | \mathbb{P}_0 \nobracket) \right\},
     \end{array} \]
  where now $\mathbb{P} \mathbb{}_{t, 0} =\mathcal{N}_{\mathbb{T}} (t, 0)
  \infixand \mathbb{P}_{t, x} =\mathcal{N}_{\mathbb{T}} (t, x)$. Let us denote
  by $\mathbb{Q}_{t, x}$ the optimizer of the first supremum (which exists and
  is explicitly given by the Radon-Nikodym density $\frac{e^{\overline{h}
  (B_t)}}{\mathbb{E}_{\mathbb{P}_{t, x}} \left[ e^{\overline{h} (B_t)}
  \right]}$). Since $\mathbb{P}_{t, 0} \sim \mathbb{P}_{t, x}$, we can also
  choose $\mathbb{Q}_{t, x}$ in the second supremum to obtain
  \[ \begin{array}{lll}
       h (t, x) - h (t, 0) & \leqslant & \mathbb{E}_{\mathbb{Q}_{t, x}} \left[
       \overline{h} (B_t) \right] - \mathbb{E}_{\mathbb{Q}_{t, x}} \left[
       \overline{h} (B_t) \right] + H (\mathbb{Q}_{t, x} | \mathbb{P}_{t, 0}
       \nobracket) - H (\mathbb{Q}_{t, x} | \mathbb{P}_{t, x} \nobracket)\\
       & = & H (\mathbb{Q}_{t, x} | \mathbb{P}_{t, 0} \nobracket) - H
       (\mathbb{Q}_{t, x} | \mathbb{P}_{t, 0} \nobracket) -
       \mathbb{E}_{\mathbb{Q}_{t, x}} \left[ \log \frac{\mathd \mathbb{P}_{t,
       0}}{\mathd \mathbb{P}_{t, x}} \right] .
     \end{array} \]
  We know that $\frac{\mathd \mathbb{P}_{t, 0}}{\mathd \mathbb{P}_{t, x}} (y)
  = \frac{\mathd \mathcal{N}_{\mathbb{T}} (t, 0)}{\mathd
  \mathcal{N}_{\mathbb{T}} (t, x)} (y) = \frac{p_{\mathbb{T}} (t,
  y)}{p_{\mathbb{T}} (t, y - x)}$, and
  \begin{equation}
    \begin{array}{l}
      | \log p_{\mathbb{T}} (t, y) - \log p_{\mathbb{T}} (t, y - x) |
      \leqslant \left\| \frac{\partial_x p_{\mathbb{T}} (t,
      \cdummy)}{p_{\mathbb{T}} (t, \cdummy)} \right\|_{\infty} | x | \leqslant
      C_T \frac{| x |}{t},
    \end{array}
  \end{equation}
  where the second inequality follows from Corollary \ref{cor:gradest} below.
  Since $h (t, x) - h (t, 0) > 0$ without loss of generality (otherwise
  exchange the role of $x$ and $0$), we deduce that
  \[ | h (t, x) - h (t, 0) | \leqslant \mathbb{E}_{\mathbb{Q}_{t, x}} \left[
     \left| \log \frac{\mathd \mathbb{P}_{t, 0}}{\mathd \mathbb{P}_{t, x}}
     \right| \right] \leqslant C_T \frac{| x |}{t} . \]
\end{proof}

In the proof we used the following results for the periodic heat kernel:

\begin{lemma}
  $\left( \cite{heattorus}, \tmop{Theorem} 2.3 \right) .$\label{heattorus}The
  heat kernel on the torus
  \[ \begin{array}{llll}
       p_{\mathbb{T}} : & \mathbb{R}^+ \times \mathbb{R} & \rightarrow &
       \mathbb{R}\\
       & (s, y) & \mapsto & p_{\mathbb{T}} (s, y) \assign 2 \pi \sum_{k \in
       \mathbb{Z}} \frac{1}{\sqrt{2 \pi s}} e^{- \frac{| y + 2 \pi k |^2}{2
       s}}
     \end{array} \]
  (or more precisely, its $2 \pi$-periodic extension) satisfies
  \[  \sqrt{\frac{2 \pi}{s}} \exp \left( - \frac{| y |_{\mathbb{T}}^2}{2 s}
     \right) \leqslant p_{\mathbb{T}} (s, y) \leqslant 2 \left( 1 +
     \sqrt{\frac{2 \pi}{s}} \right) \exp \left( - \frac{| y
     |_{\mathbb{T}}^2}{2 s} \right), \]
  where $| y |_{\mathbb{T}} \assign \underset{k \in \mathbb{Z}}{\inf}  | y - 2
  \pi k |$.
\end{lemma}

\begin{corollary}
  \label{cor:gradest}For $p_{\mathbb{T}} : [0, T] \times \mathbb{T}
  \rightarrow \mathbb{R}$ the following gradient estimate holds:
  \[ \left\| \frac{\partial_x p_{\mathbb{T}} (t, \cdummy)}{p_{\mathbb{T}} (t,
     \cdummy)} \right\|_{L^{\infty} (\mathbb{T})} \leqslant \frac{C_T}{t} . \]
\end{corollary}

\begin{remark}
  \label{harnack}We can extend this without much work to the case of bounded
  continuous forcing $\xi \in C_b ([0, T] \times \mathbb{T})$. It is also
  possible to prove, by a similar argument, that if $0 < \tau < T$
  \[
     \sup_{\substack{x,y \in \mathbb{T}\\ \tau \leqslant s \leqslant t \leqslant T}} |h(t,x) - h(s,y)| \leqslant C(\xi,\tau,T),
  \]
  independently of the initial condition. Taking $w = e^h$, that is, the
  solution of
  \[ \left( \partial_t - \frac{1}{2} \Delta \right) w = \xi w, \quad w (0) =
     e^{\bar{h}}, \]
  this yields
  \begin{equation}
    \underset{(t, x) \in [\tau, T] \times \mathbb{T}}{\sup} w (t, x) \leqslant
    K (\xi, \tau, T) \underset{(t, x) \in [\tau, T] \times \mathbb{T}}{\inf} w
    (t, x), \label{harnacktorus}
  \end{equation}
  i.e., Harnack's inequality for the heat equation (on the torus). Inequality
  \eqref{harnacktorus} does not hold in general in the whole space  R, but if
  $0 < t_1 < T_1 < t_2 < T_2$ and $M > 0$, it is true that
  \[ \underset{(t, x) \in [t_1, T_1] \times [- M, M]}{\sup} w (t, x) \leqslant
     K (\xi, M, t_1, T_1, t_2, T_2) \underset{(s, y) \in [t_2, T_2] \times [-
     M, M]}{\inf} w (s, y), \]
  which we can also derive as in Theorem \ref{oscitheorem} by noting that if
  $0 < t_1 \leqslant t \leqslant T_1 < t_2 \leqslant s \leqslant T_2$ and $x
  \in [- M, M]$,
  \[ \frac{\mathd \mathcal{N}_{\mathbb{R}} (s, 0)}{\mathd
     \mathcal{N}_{\mathbb{R}} (t, x)} (y) = \frac{p_{\mathbb{R}} (s,
     y)}{p_{\mathbb{R}} (t, y - x)} \geqslant C (M, t_1, T_1, t_2, T_2), \]
  where it is essential that $t \leqslant T_1 < t_2 \leqslant s$.
\end{remark}

\subsection{The case of space-time white noise forcing}

\begin{theorem}
  \label{roughoscitheorem}Let $\bar{h} \in \mathcal{C}^{\beta} (\mathbb{T})$
  for some $\beta > 0$. Then, for any $\alpha \in \left( 0, \frac{1}{2}
  \right)$ and some positive constant
  \[ C \equiv C (\alpha, t, T, Y, Y^{\vee}, Y^R), \]
  the solution of the KPZ equation satisfies
  \begin{equation}
    | h (t, x) - h (t, 0) | \leqslant C | x |^{\alpha}, \quad (t, x) \in (0,
    T] \times \mathbb{T}, \label{roughosci}
  \end{equation}
  independently of the initial condition, where $\sup_{t \in [\varepsilon, T]}
  C (t, T, \alpha, \xi) < \infty$for any $\varepsilon > 0$.
\end{theorem}

\begin{proof}
  In a first step, we remove the forcing: consider $\hat{Y}$ such that
  \[ \left( \partial_t - \frac{1}{2} \Delta \right) \hat{Y} = \frac{1}{2} |
     \partial_x Y^R |^2 + \partial_x (Y + Y^{\vee} + Y^R) \partial_x  \hat{Y}
     + \frac{1}{2} | \partial_x \hat{Y} |^2, \quad \hat{Y} (0) = 0, \]
  which only depends on $Y$, $Y^{\vee}$ and $Y^R$, but not on $\bar{h}$. Then,
  $u \assign h^R - \hat{Y}$ (see Remark \ref{enhancednoise}) satisfies
  \[ \left( \partial_t - \frac{1}{2} \Delta \right) u = \partial_x (Y +
     Y^{\vee} + Y^R + \hat{Y}) \partial_x u + \frac{1}{2} | \partial_x u |^2,
     \quad h^R (0) = \bar{h} - Y (0), \]
  and we can prove in the same way as in Theorem \ref{varirep} that
  \[ u (t, x) = \underset{v \in \mathfrak{p}\mathfrak{m}}{\sup}
     \underset{\gamma \in \mathfrak{M} (v, x)}{\sup} \mathbb{E} \left[
     \overline{h} (\gamma_t) - Y (0, \gamma_t) - \frac{1}{2} \int^t_0 | v_s
     (\gamma) |^2 \mathd s \right], \]
  where $\gamma$ is now given by
  \begin{equation}
    \gamma_s = x + \int^s_0 v_r \mathd r + \int^s_0 \partial_x (Y + Y^{\vee} +
    Y^R + \hat{Y}) (t - r, \gamma_r) \mathd r + B_s . \label{newsingdiff}
  \end{equation}
  We can proceed as in Theorem \ref{oscitheorem} (as Lemma
  \ref{singdupuistheorem} and Corollary \ref{singkernelcorol} still hold) to
  conclude that
  \[ | u (t, x) - u (t, 0) | \leqslant | \log \tilde{p} (0, x ; t, y) - \log
     \tilde{p} (0, 0 ; t, y) | \leqslant \underset{y \in \mathbb{T}}{\sup}
     \left\| \frac{\partial_x \tilde{p} (0, \cdot ; t, y)}{\tilde{p} (0, \cdot
     ; t, y)} \right\|_{\infty} | x |, \]
  where $\tilde{p}$ denotes the (periodised) transition density of
  \eqref{newsingdiff} with $v = 0$, so overall we have, since
  \[ u = h - Y - Y^{\vee} - Y^R - \hat{Y}, \]
  that
  \[ | h (t, x) - h (t, 0) | \leqslant \left( \| Y + Y^{\vee} + Y^R + \hat{Y}
     \|_{C_T \mathcal{C}^{\alpha}} + C_{\alpha} \underset{y \in
     \mathbb{T}}{\sup} \left\| \frac{\partial_x \tilde{p} (0, \cdot ; t,
     y)}{\tilde{p} (0, \cdot ; t, y)} \right\|_{\infty} \right) | x |^{\alpha}
     . \]
  By working in Besov spaces on the $L^1$ scale and using that the Dirac delta
  is in $B^0_{1, \infty} (\mathbb{T})$, we get $\tilde{p} (0, \cdummy ; t, y)
  \in B^{3 / 2 -}_{1, \infty} (\mathbb{T})$ for all $t > 0$. Going until time
  $t / 2$ and embedding $B^{3 / 2 -}_{1, \infty} (\mathbb{T}) \subset
  B^0_{\infty, \infty} (\mathbb{T})$ and then solving the equation on the
  $L^{\infty}$ scale, we get that $\tilde{p} (0, \cdummy ; t, y) \in C^{3 / 2
  -} (\mathbb{T})$; see also Section 2 of {\cite{nicolaswillem}} for a similar
  argument. By the heat kernel above $\tilde{p} (0, \cdot ; t, y)$ is
  uniformly lower bounded with bound depending on $t$, so overall
  \[ \underset{y \in \mathbb{T}}{\sup} \left\| \frac{\partial_x \tilde{p} (0,
     \cdot ; t, y)}{\tilde{p} (0, \cdot ; t, y)} \right\|_{\infty} = C (t) <
     \infty, \]
  and $C (t)$ is uniformly bounded on $[\varepsilon, T]$ for any $\varepsilon
  > 0$.)
\end{proof}

\begin{remark}
  If $h$ solves the KPZ equation, its Cole-Hopf transform $w = e^h$ solves (in
  a suitable sense; see Lemma 4.7 in {\cite{KPZreloaded}}) the rough heat
  equation
  \[ \left( \partial_t - \frac{\Delta}{2} \right) w = w \xi, \quad w (0) =
     e^{\bar{h}} . \]
  In the spirit of Remark \ref{harnack}, we can thus obtain the Harnack type
  inequality
  \[ \underset{(t, x) \in [\tau, T] \times \mathbb{T}}{\sup} w (t, x)
     \leqslant K (Y, Y^{\vee}, Y^R, \tau, T) \underset{(t, x) \in [\tau, T]
     \times \mathbb{T}}{\inf} w (t, x), \]
  for the rough heat equation on the torus.
\end{remark}

\bibliographystyle{alpha}
	\bibliography{biblio.bib}

\end{document}